\documentclass[a4paper,11pt]{article}
\usepackage{latexsym,amsmath,amssymb, amsfonts,mathrsfs}
\usepackage{microtype}
\usepackage{graphicx}
\usepackage[arrow,matrix,curve]{xy}
\usepackage{color}
\usepackage{a4wide}
\usepackage[normalem]{ulem}
\usepackage{hyperref}
\usepackage{wasysym}
\usepackage{pdfsync}
\usepackage{comment}
\usepackage{rotating}
\definecolor{darkblue}{rgb}{0.0,0.0,1.0}
\definecolor{darkgreen}{rgb}{0.1,0.6,0.0}
\definecolor{darkred}{rgb}{1.0,0.0,0.6}
\hypersetup{
     colorlinks=true,         
     linkcolor=black,
     citecolor=black,
}

 1

\usepackage{amsthm}

\theoremstyle{plain}
\newtheorem{Theorem}{Theorem}[section]
\newtheorem{Lemma}[Theorem]{Lemma}
\newtheorem{Proposition}[Theorem]{Proposition}

\theoremstyle{definition}
\newtheorem{Definition}[Theorem]{Definition}
\newtheorem{Example}[Theorem]{Example}

\newtheorem{Remark}[Theorem]{Remark}

\numberwithin{equation}{section}

\newcommand{\eq}{\begin{equation}}
\newcommand{\eqa}{\begin{eqnarray}}
\newcommand{\en}{\end{equation}}
\newcommand{\ena}{\end{eqnarray}}
\newcommand{\enn}{\nonumber \end{equation}}

\def\sk{\vskip .4cm}
\def\noi{\noindent}
\def\de{\delta}
\def\epsi{{\varepsilon}}
\def\st {\star}
\def\f{{\rm{f}\,}}
\def\of{{\overline{{\rm{f}}\,}}}

\def\D/h{\widehat{\fmslash D}}

\def\al{\alpha}

\def\be{\beta}
\def\ga{\gamma}

\def\de{\delta}

\def\5bar{{\overline 5}}

\def\MMM{{\mathscr M}^{}}
\def\RR{{\mathcal R}}

\newcommand{\MMMod}[3]{{}^{#1}{}_{#2}\MMM{\!}_{#3}}
\newcommand{\AAAlg}[3]{{}^{#1}{}_{#2}\AAA{\!}_{#3} }

\def\R{{R}}
\def\oR{{\overline{\R}}}

\def\gg{{\hat g}}

\def\UU{H}

\def\FF{\mathcal F}

\def\varepsi{\varepsilon}

\def\s'O{\stackrel{_{{\displaystyle\st \footnotesize '}}}{_{^{^{\displaystyle\otimes}}}}}

\def\v{\chi}

\def\AAs{{A_\st }}

\def\D{\Delta}
\def\1s{{1_\st }}
\def\3s{{3_\st }}
\def\2s{{2_\st }}
\def\ef1{{1_\FF}}
\def\ef2{{3_\FF}}
\def\ef3{{2_\FF}}

\def\hbar{\lambda}

\def\le{\langle}
\def\re{\rangle}
\def\nn{\nonumber}

\def\bbC{\mathbb{C}}
\def\bbK{\mathbb{K}}
\def\bfK{\mathbb{K}}
\def\bbA{\mathbb A}

\def\AA{A}
\def\AAA{\mathscr A}
\def\dd{{\nabla}}
\def\dif{{\mathrm d}}

\def\gg{\mathfrak{g}}

\def\trgl{\triangleright}
\def\btrgl{\blacktriangleright}

\def\Hom{\text{\sl Hom}}
\def\Con{\text{\sl Con}}
\def\End{\text{\sl End}}

\def\ra{\triangleright}
\def\RA{\blacktriangleright}
\def\pP{P}
\def\qQ{Q}

\def\id{\mathrm{id}}

\begin{document}
\thispagestyle{empty}
\begin{flushright}
\small{\rightline{CERN-PH-TH/2012-064}}
\end{flushright}

\boldmath
\begin{center}
{\Large {\bf 
Twisting all the way: \\[0.4em] from algebras to morphisms and
connections\footnote{Paper based on joint work with Alexander Schenkel}}}
\sk\sk
{\bf Paolo Aschieri}\\[0.2em]
{\sl Physics Department, Theory Unit, CERN, CH 1211, Geneva 23, Switzerland
\\and\\
 Dipartimento di Scienze e Innovazioni Tecnologiche, \\
Universit\`a del Piemonte Orientale and INFN gruppo collegato di
Alessandria, \\Viale T. Michel 11, I-15121, Alessandria, Italy}\end{center}

\vspace{-0.4cm}
\unboldmath
\vspace{0.8cm}
\begin{abstract}
Given a Hopf algebra $\UU$ and an algebra $\AA$ that is an $\UU$-module
algebra we consider the category of left $\UU$-modules
and $\AA$-bimodules $\MMMod{H\!\!\!}{A}{A}$, where morphisms are just right $A$-linear maps
(not necessarily $H$-equivariant). Given a twist $\FF$ of $\UU$ we
then quantize (deform) $\UU$ to
$\UU^\FF$, $A$ to $A_\st$ and correspondingly the category $\MMMod{H\!\!\!}{A}{A}$ to
$\MMMod{H^\FF\!\!\!\!}{A_\st}{A_\st}$.
If we consider a quasitriangular  Hopf algebra
$\UU$, a quasi-commutative algebra $A$ and quasi-commutative
$A$-bimodules, we can further
construct and study tensor products over $A$ of modules and of morphisms, and
their twist quantization.

This study leads to the definition of {arbitrary} (i.e., not necessarily
$H$-equivariant) connections on quasi-commutative $A$-bimodules,
to extend these connections to tensor product modules and to quantize them to $A_\st$-bimodule connections. 
Their curvatures and those on 
tensor product  modules are also determined.
\end{abstract}
\sk\sk\sk\small{\noi keywords: 
Noncommutative Geometry,  Drinfeld Twist, bimodule connections, Universal deformation fomula}

\paragraph*{MSC 2010:}46L87,  17B37,   53D55,     81R60

\tableofcontents



\section{Introduction}

Consider the basic algebraic structures underlying the differential geometry of a
manifold $M$:
the algebra $A=C^\infty(M)$ of complex valued functions on $M$;
the $A$-module of sections of the tangent bundle 
(i.e., vector fields),  and that  of one-forms;
the algebra of tensor fields $({\cal T},\otimes_A)$ and the exterior
algebra $(\Omega^\bullet, \wedge)$. 
Typical maps between these algebraic structures are 
the exterior derivative, connections, and tensor fields (like a
metric tensor, a curvature tensor, etc.). 
The Lie algebra of vector fields (infinitesimal local diffeomorphisms
of $M$) acts on all the above structures. The universal
enveloping algebra of the Lie algebra of vector fields is a
Hopf algebra $H$ and it also acts on all the above structures.

In \cite{Aschieri:2005yw, Aschieri:2005zs} we have deformed the Hopf algebra 
$H$ of vector fields on a manifold $M$ via a Drinfeld twist 
(or twisting element) $\FF\in H\otimes H$ \cite{Drinfeld:1989st}. 
Since the $A$-modules of vector fields and of one-forms, and 
the  tensor and exterior
algebras $({\cal T},\otimes_A)$ and  $(\Omega^\bullet, \wedge)$ carry a representation of 
the algebra $H$ we have been able to deform these algebraic structures as well.
Concerning morphisms, in \cite{Aschieri:2005yw, Aschieri:2005zs} we study in particular the
deformations of the Lie derivative and inner derivative. 
Since $H$ is the Hopf algebra associated to the Lie algebra of vector
fields, interesting morphisms are not in general $H$-equivariant
(they indeed transform covariantly).  
This deformed geometry has then been used to formulate a noncommutative
gravity theory.  Notice that the noncommutative connection
and metric tensor considered in \cite{Aschieri:2005yw, Aschieri:2005zs} cannot be
$H$-equivariant because they are dynamical fields.
  \sk
In the present paper, outlining joint work with Alexander Schenkel \cite{AS}, 
we clarify and further study the twist quantization or
twist deformation scheme. 
Even if our leading example is the twist deformation of the algebraic
structures of a manifold $M$, we here more in general consider a Hopf
algebra $H$ and  $A$-modules where $H$ is not necessarily
cocommutative like the universal enveloping algebra of a Lie
algebra (for example it can be a quantum group)
and $A$ is not necessarily commutative.

More precisely in this paper  we study  the category of $A$-modules carrying an
action of a Hopf algebra $H$. The obvious morphisms in this category are
$A$-module morphisms that are also $H$-equivariant ($H$-module
morphisms).  As we have seen, this choice is too restrictive
and we therefore study the category $\MMMod{H\!\!\!}{}{A}$ of right $A$-modules
that carry also an $H$-action, but where morphisms are just right $A$-linear maps. 
We will also study just linear (not $A$-linear) maps, this is
propaedeutical and it is also needed in order to understand deformations of connections.
We further study the category  $\MMMod{H\!\!\!}{A}{A}$ of left and right $A$-modules ($A$-bimodules)
so that we can consider the tensor product over $A$ of $A$-bimodules
(the algebraic analogue of tensor product of vector bundles).
In the category $\MMMod{H\!\!\!}{A}{A}$ morphisms are just right
$A$-module morphisms; their tensor product
will
be defined for a subclass of noncommutative algebras $A$ and
$A$-bimodules: quasi-commutative algebras and  
bimodules carrying an action of a quasitriangular Hopf algebra
$H$. Examples of such modules with a triangular Hopf algebra action
naturally arise when considering twist deformations of commutative
algebras and cocommutative Hopf algebras.

All these structures can be deformed  via a Drinfeld twist $\FF\in H\otimes H$
that in particular acts nontrivially on morphisms. If the algebra $A$
is commutative then the twist deformed algebra $A_\st$ is usually
noncommutative, it is a quantization of $A$. We call 
twist quantization or simply  quantization the twist deformation of the
algebra $A$ to $A_\st$, and of the category $\MMMod{H\!\!\!}{A}{A}$ to
$\MMMod{H^\FF\!\!\!}{A_\st}{A_\st}$, also when $A$ is not commutative.

Having understood the properties of linear and $A$-linear morphisms we can
then present our theory of  connections. These are just right
$A$-module connections. Quasitriangularity of the Hopf algebra $H$ and
quasi-commutativity of  the $A$-bimodules then imply that these
connections are also quasi-left $A$-module connections. 
These connections are more general than the ones studied
in the literature on $A$-bimodules \cite{Mourad:1994xa, DuboisViolette:1995hh, Madore:2000aq}.
In the restrictive case they are also $H$-equivariant then
they match 
the definition in \cite{Mourad:1994xa, DuboisViolette:1995hh, Madore:2000aq}.

Given two connections on two quasi-commutative $A$-bimodules $V$ and
$W$ we  next define their sum as the connection on the tensor product
module $V\otimes_A W$ (physically this is relevant for example when considering the covariant
derivative of composite fields). Again this differs from the sum of
bimodule connections discussed in the literature
\cite{Bresser:1995gk, Madore:2000aq, DuboisViolette:1999cj}.
We further  study the
corresponding quantized connections as well as their curvatures. 
Twist deformation acts nontrivially on these structures, for example
flat connections are twisted in non flat ones and vice versa.


\section{Twisting algebras, modules and morphisms}
\subsection{Hopf algebras, modules, twists; twisted algebras and modules}
We start settling the notation and recalling the structures we will
later deform.
In deformation quantization the field of complex numbers
$\bbC$ is replaced by the ring $\bbC[[h]]$ of formal power series 
(in an indeterminate, say $h$) with
coefficients in $\bbC$.  In order to cover also this example (that is
rich of twists, an example being the Moyal-Weyl twist) we
consider modules and algebras over a commutative ring $\bbK$ with unit element $1\in\bbK$. 
Then vector spaces over $\bbC$ are replaced by $\bbK$-modules, and
$\bbC$-linear maps are replaced by {$\bbK$-linear maps} ({$\bbK$-module morphisms}).

We recall that a Hopf algebra $H$ is an algebra (over $\bbK$) with
multiplication $\mu: H\otimes H\rightarrow H$ and two algebra morphisms $\Delta:\UU\to\UU\otimes\UU$ (coproduct), $\varepsilon:\UU\to\bbK$ (counit) and a
$\bbK$-linear map $S:\UU\to\UU$ (antipode) satisfying,
for all $\xi\in\UU$,
$(\Delta\otimes\id)\Delta(\xi)= (\id\otimes\Delta)\Delta(\xi)$ and
\eq
(\varepsilon\otimes\id)\Delta(\xi)=(\id\otimes \varepsilon)\Delta(\xi) = \xi~~,~~~
\mu\bigl((S\otimes\id)\Delta(\xi)\bigr) = \mu\bigl((\id\otimes S)\Delta(\xi)\bigr)=\varepsilon(\xi) 1~,
\en
where $1$ is the unit in $H$.
As usual we use Sweedler's notation for the coproduct, for all $\xi\in\UU$,
$\Delta(\xi)=\xi_{1}\otimes\xi_{2}$ (sum understood), similarly 
$(\Delta\otimes \id)\Delta(\xi)=(\id\otimes
\Delta)\Delta(\xi)=\xi_{1}\otimes\xi_{2}\otimes \xi_3$.

\sk
A left module $V$ over an algebra $\AA$ (or {\bf left
$\AA$-module}) is a $\bbK$-module $V$ with a 
$\bbK$-linear map $\cdot:\AA\otimes V\to V$ satisfying, for all $a,a'\in\AA$ and $v\in V$,
\begin{flalign}
(a\, a') \cdot v = a\cdot(a'\cdot v)~,~~1\cdot v = v~.
\end{flalign} 
The map $\cdot:\AA\otimes V\to V$ is called an action of
$A$ on $V$ or a representation of $A$ on $V$.
We denote the category of  left $\AA$-modules  by $\MMMod{}{A}{}$.
Morphisms in this category are left $A$-linear maps. Similarly $\MMMod{}{}{A}$
denotes the category of right $A$-modules.
A left $A$-module and a right $A$-module structure on $V$ are compatible if left and
right $A$-actions commute,  
for all $a, a'\in A$, and $v\in V$,
$(a\cdot v)\cdot a' =a\cdot (v\cdot a')~.$
Then $V$ is an $A$-bimodule. We denote the corresponding category  by $\MMMod{}{A}{A}$.
\sk
A {\bf left $\UU$-module algebra} is an algebra $\AA$ which is also a left $\UU$-module (where
the left action is denoted by $\trgl$),
such that for all $\xi\in\UU$ and $a,b\in\AA$,
\begin{flalign}
\xi \trgl(a\,b)=(\xi_1\trgl a)\,(\xi_2\trgl b)~,~~\xi\trgl1 =\varepsilon(\xi)\,1~.
\label{coproductonab}\end{flalign}
We define $\AAAlg{H\!\!\!}{}{}$ to be the category of $\UU$-module
algebras where morphisms are algebra morphisms $\rho: A\to \tilde A$
that are {\it not} necessarily also $H$-module
morphisms (i.e. $H$-equivariant):
for all $\xi\in H,\, a\in A$, $\rho(\xi\trgl a)=\xi\,\tilde\trgl\, \rho(a)$.

We can now consider $A$-bimodules $V$, where $A\in \AAAlg{H\!\!}{}{}$ and $V$ is
also a left $H$-module. Compatibility between the Hopf algebra structure
of $H$ and the $A$-bimodule structure of $V$  leads to the following
covariance requirement:
\begin{Definition}
A {\bf left $\UU$-module $A$-bimodule $V$} is an $A$-bimodule $V$ over $\AA\in \AAAlg{H\!\!}{}{}$
which is also a left $\UU$-module, such that
for all $\xi\in\UU$, $a\in A$ and $v\in V$,
\eq\label{eqn:moduleHcovariance}
 \xi\trgl(a\cdot v) = (\xi_1\trgl a)\cdot (\xi_2\trgl v)~,~~
 \xi\trgl(v\cdot a) = (\xi_1\trgl v)\cdot (\xi_2\trgl a)~.
\en
The category of left $\UU$-module $\AA$-bimodules is denoted by $\MMMod{H\!\!\!}{A}{A}$.
We define morphisms in this category to be right $A$-module
morphisms but not necessarily $H$-module morphisms or left $A$-module morphisms.
The subcategory of $\MMMod{H,\!\!\!}{A}{A}$ given by modules that are
also $H$-module algebras, and morphisms that are algebra morphisms is
denoted by  $\AAAlg{H\!\!}{A}{A}$.
\end{Definition}

In commutative differential geometry, as discussed in the
introduction, we encounter all these structures. For example we have
that the bimodule of one forms $\Omega$ on a manifold $M$ is an
element in the category
$\MMMod{U\Xi}{C^\infty(M)}{C^\infty(M)}$,
where $U\Xi$ is the universal enveloping algebra of the Lie algebra of
vector fields and their  action is given by the Lie derivative.

\sk
Following Drinfeld \cite{Drinfeld:1989st} and Giaquinto and Zhang \cite{Giaquinto:1994jx} 
we now  deform these modules and algebras via a twist 
\begin{Definition}\label{twistdeff}
A {\bf twist }
$\FF\in \UU\otimes \UU$ of a Hopf algebra $H$ is an invertible element that satisfies
\begin{subequations}
\label{propF}
\begin{flalign}
\label{propF1}
\FF_{12}(\Delta\otimes \id)\FF &=\FF_{23}(\id\otimes \Delta)\FF~,\quad
\text{($_{\,}2$-cocycle property)}\\
\label{propF2}
(\varepsilon\otimes \id)\FF &=1=(\id\otimes \varepsilon)\FF~,\quad 
\text{(normalization property)}
\end{flalign}
\end{subequations}
where $\FF_{12}=\FF\otimes 1$ and $\FF_{23}=1\otimes \FF$.
\end{Definition}
We shall frequently use the notation (sum over $\al$ understood)
\eq\label{Fff}
\FF=\f^\al\otimes\f_\al~~~,~~~~\FF^{-1}=\of^\al\otimes\of_\al~.
\en
In order to get familiar with this notation we rewrite
the inverse of the $2$-cocycle condition (\ref{propF1}), 
$
((\Delta\otimes\id)\FF^{-1}) 
\FF^{-1}_{12} =((\id \otimes \Delta)\FF^{-1})\FF^{-1}_{23},
$
in this notation,
\begin{equation}
\label{ass}
\of_{_1}^\al\of^\be\otimes \of_{_2}^\al\of_\be\otimes \of_\al=
\of^\al\otimes {\of_{\al_1}}\of^\be\otimes {\of_{\al_2}}\of_\be~.
\end{equation}
In deformation quantization the twist is $\FF=1\otimes1
+hF_1+h^2F_2+...$ where $F_i\in H\otimes H$, and leads to a formal
deformation quantization in the parameter $h$, see Theorem \ref{Theorem1}. 

\begin{Theorem}\label{TwistedHopfAlg}
The twist ${\cal F}$ of the Hopf algebra $\UU$ leads to a new 
 Hopf algebra $\UU^\FF$,  given by
\begin{flalign}
(\UU ,\mu, \Delta^{\cal F}, S^{\cal F}, \varepsilon)~.
\end{flalign}
As algebras $\UU^\FF =\UU$ and they also have the same counit 
$\varepsilon^\FF=\varepsi$. 
The new coproduct $\Delta^\FF$ is given by, for all $\xi \in \UU$,
$
\Delta^{\FF}(\xi) = {\FF}\Delta(\xi){\FF}^{-1}\,.
$
The new antipode is
$
S^\FF(\xi)=\chi S(\xi)\chi^{-1}~,
$
where
$\chi := \f^\al S(\f_\al) \,,~\chi^{-1} = S(\of^\al) \of_\al\,.$
\end{Theorem}

\begin{Theorem}\label{Theorem1}
Given a Hopf algebra $H$, 
a twist $\FF\in H\otimes H$ and  a  left $H$-module algebra $\AA$ (not
necessarily associative or with unit), then there exists
a deformed  left $H^\FF$-module algebra $\AA_\st$. 
The algebra $\AA_\st$ has the same $\bbK$-module
structure as $\AA$ and
the action of $H^\FF$ on $\AA_\st$
is that of $H$ on $\AA$.
The product in $\AA_\st $ is defined by, for all $a,b\in A$,
\eq
a\st  b :=\mu\circ \FF^{-1}\trgl(a\otimes b)=(\of^\al\trgl
a)\,(\of_\al\trgl b)~.
\en
If $\AA$ has a unit element then $\AAs$ has the same unit element. If 
$\AA$ is associative then $\AA_\st $ is an associative algebra as well.
\end{Theorem}
\begin{proof}
We have to prove  that the product in $\AA_\st$ 
is compatible with the Hopf algebra structure on $H^\FF$, for all $a,b\in A$ and $\xi\in H$,
\eqa
\xi\trgl(a\st b)&=&\xi\trgl\bigl(\mu\circ \FF^{-1}\trgl (a\otimes b)\bigr)\nn
=
\mu\circ \D(\xi)\trgl \circ \,\FF^{-1}\trgl (a\otimes b)\nn
=
\mu\circ(\D(\xi)\, \FF^{-1})\trgl  (a\otimes b)\\
&=&\mu\circ\FF^{-1}\trgl\circ\D^\FF(\xi)\trgl(a\otimes b)
=(\xi_{1_\FF}\trgl a)\st(\xi_{2_\FF}\trgl b)~,\label{ccc}
\ena
where we used the notation $\Delta^\FF(\xi)=\xi_{1_\FF}\otimes \xi_{2_\FF}$.

If $\AA$ has a unit element $1$, then  $1\st a=a\st1=a$ follows from the
normalization property (\ref{propF2}) of the twist $\FF$.
If $\AA$ is an associative algebra we also have to prove
associativity of the new product, for all $a,b,c\in A$,
\begin{flalign}
\nn (a\st b)\st c&=\of^\al\trgl \bigl((\of^\be\trgl a)(\of_\be\trgl
b)\bigr)\;(\of_\al\trgl c)\\
&=
(\of^\al\trgl a)\,\of_\al\trgl \bigl((\of^\be\trgl b)(\of_\be\trgl
c)\bigr)=a\st (b\st c)~,
\end{flalign}
where we used the 2-cocycle property  (\ref{propF1})  of the twist in the notation of (\ref{ass}).
\end{proof}

\begin{Theorem}\label{Theorem2}
In the hypotheses of Theorem \ref{Theorem1}, given a left $\UU$-module
$\AA$-bimodule $V\in \MMMod{H\!\!\!}{A}{A}$, then there exists
a left $\UU^\FF$-module $\AA_\st$-bimodule $V_\st\in \MMMod{H^\FF\!\!\!\!}{A_\st}{A_\st}$.
The module $V_\st$ has the same $\bbK$-module
structure as $V$ and the left action of $H^\FF$ on $V_\st$
is that of $H$ on $V$. The $A_\st$ actions on $V_\st$
are respectively defined by, for all $a\in A$ and $v\in V$,
\begin{subequations}
\label{modst}
\begin{flalign}
a\st  v &=\cdot\circ \FF^{-1}\trgl(a\otimes v)=(\of^\al\trgl
a)\cdot (\of_\al\trgl v)~, \label{lmodst}\\
v\st  a &=\cdot\circ \FF^{-1}\trgl(v\otimes a)=(\of^\al\trgl
v)\cdot(\of_\al\trgl a)~. \label{rmodst}
\end{flalign}
 \end{subequations}
If $V$ is further a left $H$-module $A$-bimodule algebra 
$V=E\in  \AAAlg{H\!\!}{A}{A}$, then $E_\st\in   \AAAlg{H^\FF\!\!\!\!}{A_\st}{A_\st}$,
where the $\st$-product in the algebra $E_\st$ is given in Theorem \ref{Theorem1}.
\end{Theorem}
\begin{proof}
Left $\AA_\st$-module property: 
\eq\label{astbstv} 
(a\st b)\st v=\of^\al\trgl \bigl((\of^\be\trgl a)(\of_\be\trgl
b)\bigr)\cdot(\of_\al\trgl v)=
(\of^\al\trgl a)\cdot \of_\al\trgl \bigl((\of^\be\trgl b)\cdot (\of_\be\trgl
v)\bigr)=a\st (b\st v)~.
\en 
The right $A_\st$-module and $A_\st$-bimodule
properties are similarly proven.
Compatibility between the left $H^\FF$ and the left $A_\st$-action, $
\xi\trgl (a\st v)=(\xi_{1_\FF}\trgl a)\st(\xi_{2_\FF}\trgl v)\,,$ 
is proven as in (\ref{ccc}). Also
the left $H^\FF$ and the right $A_\st$-action compatibility is
similarly proven.

In case we  have $V=E\in    \AAAlg{H\!\!}{A}{A}$, then $E_\st\in  \AAAlg{H^\FF\!\!\!\!}{A_\st}{A_\st}$
 because of Theorem \ref{Theorem1}.
\end{proof}

We observe that these structures can  be untwisted to the
original ones. This is done using the twist $\FF^{-1}$ of the Hopf
algebra $H^\FF$.
Moreover, if we consider only $H$-equivariant morphisms then these are not
deformed by the twisting procedure. Then Theorem {\ref{Theorem2}}
states that the category  $\MMMod{H\!\!\!}{A}{A}^{\,eqv}$ and the deformed category
$\MMMod{H^\FF\!\!\!\!}{A_\st}{A_\st}^{\,eqv}$ of $A$-bimodules with
$H$-equivariant morphisms are equivalent \cite{Giaquinto:1994jx};
(this result  follows from the equivalence of the tensor
categories of $H$-modules and twisted $H$-modules \cite{Drinfeld:1989st}).


\subsection{Twisting morphisms: the quantization map $D_\FF$}

The action $\ra$ of an Hopf algebra $H$ on an $H$-module $V$ lifts to
$\End_\bbK(V)$ (the algebra of endomorphisms or $\bbK$-linear maps $V\to V$), via the adjoint action,
for all $\xi\in H$ and $P\in \End_\bbK(V)$,
\eq
\xi\btrgl P := \xi_1 \trgl\circ P \circ S(\xi_2)\trgl\label{xiadjactP}~.
\en
This gives the algebra of endomorphisms $\End_\bbK(V)$ 
an $H$-module algebra structure. 
We can consider the deformed algebra 
$\End_\bbK(V)_\st$,  with the new composition product between
endomorphisms $P,Q\in \End(V)$,
\eq
P\circ_\st Q:=(\of^\al\RA P)\circ (\of_\al\RA Q)~.
\en
The algebras $\End_\bbK(V)_\st$ and 
$\End_\bbK(V)$ are respectively $H^\FF$ and $H$-module algebras. However as algebras they are isomorphic. 
This was proven in \cite{Aschieri:2005zs} in the case $A$ is $H$
itself. The same techniques show that it holds
more in general \cite{Kulish:2010mr, AS}. 
The isomorphism is given by
\eq\label{eqn:Ddef}
D_\FF:\End(V)_\st\to\End(V)~,~~P\mapsto
D_\FF(\pP):=(\of^\al\!\RA \pP)\circ \:\of_\al\ra=\of^\al_1 \ra\circ
\pP \circ S(\of^\al_2)\of_\al\ra ~.
\en
The twist and Hopf algebras properties imply that, for all $P,Q\in \End(V)$
\eq\label{DPSTQDPDQ}
D_\FF(\pP\circ_\st \qQ)=D_\FF(\pP)\circ D_\FF(\qQ)~,
\en
or equivalently $
D_\FF\circ \mu\circ \FF^{-1}\ra= \mu \circ (D_\FF\otimes D_\FF)
$  (where $\mu(P\otimes Q)=P\circ Q$).
An equivalent expression for $D_\FF$ can be shown to be $D_\FF(\pP)
=\f^\be\ra\circ \pP\circ  S(\f_\be) \v^{-1}\ra $, where
$\v^{-1}=S(\of^\al)\of_\al$. 
The inverse of $D_\FF$ is then given by
$
{D_\FF}^{-1}(\pP)=\of^\al\ra \circ \pP\circ \v S(\of_\al)\ra,
$
where $\v=\f^\be S(\f_\be)$.

\begin{Remark} We mention a special example in which  $V$ is the algebra
$A=C^\infty(M)$ on a manifold $M$, and  $H$ is the universal
enveloping algebra  $U\Xi$ of vector fields. 
Functions $f\in A$ can be seen as $\bbK$-linear maps $f:A\to A$, $h\mapsto
f h$ for all $h\in A$. Then formula (\ref{DPSTQDPDQ}) reads, for
$f,g\in A$,
\eq
D_\FF(f\st g)=D_\FF(f)\circ  D_\FF(g)~~~\mbox{i.e.,}~~~~ 
f\st g=D_\FF^{-1}(D_\FF(f) \circ D_\FF(g))~.
\en 
In words: the star product
of functions can be obtained by mapping these to the differential
operators $D_\FF(f)$ and $D_\FF(g)$, composing, and then transforming
back to function space.
\end{Remark}

There is another route to the quantization of the $H$-module
$\End_\bbK(V)$. We first consider $V$ as an $H^\FF$-module (this is trivially so
because $H$ and $H^\FF$ are the same algebra), however, to stress that
now it belongs to the $\MMMod{H^\FF\!\!\!\!}{}{}$ category we denote it by
$V_\st$.
Then $\End_\bbK(V_\st)\in \MMMod{H^\FF\!\!\!\!}{}{}$, where the $H^\FF$
adjoint action is given by, for all $\xi\in H^\FF$ and $P\in \End_\bbK(V_\star)$,
\begin{flalign}\label{HFadJact}
\xi\btrgl_\FF P := \xi_{1_\FF}\ra\,\circ P\circ S^\FF(\xi_{2_\FF})\ra\,~.
\end{flalign}
We will frequently write $\,(\End_\bbK(V_\st),\RA_\FF)\in \MMMod{H^\FF\!\!\!\!}{}{}\,$
in order to specify the $H^\FF$-action we are considering.
\begin{Theorem}\label{PDP0}
The map  
\begin{flalign}
\label{DFFEK}
D_\FF ~:~\End_\bbK(V)_\st~~&\longrightarrow~~~~ \End_{\bbK}(V_\st)\nn\\
 P~~~~&\longmapsto~~~~  D_\FF(P):=(\of^\al\btrgl P)\circ \of_\al\trgl~ 
\end{flalign}
is an isomorphism between the left $H^\FF$-module algebras
$(\End_\bbK(V)_\st, \RA)\in \AAAlg{H^\FF\!\!\!\!}{}{}$
and $(\End_{\bbK}(V_\st),\RA_\FF)\in  \AAAlg{H^\FF\!\!\!\!}{}{}$.
We call $D_\FF(P)$ the 
quantization of the endomorphism $P$.
\end{Theorem}
\begin{proof} We already know that $D_\FF$ is an isomorphism of
  algebras. We have to check $H^\FF$-equivariance, i.e., that
 $D_\FF$ intertwines between the two  
$H^\FF$-actions,
for all $\xi\in \UU$ and $P\in \bbA$,
\eq\label{DHFintertwiner}
D_\FF(\xi \trgl \pP)=\xi\trgl_\FF D_\FF(\pP)~.
\en
Using the expression $D_\FF(\xi \trgl \pP)
=\f^\be (\xi \trgl \pP) S(\f_\be) \v^{-1}$ we have
\eqa
D_\FF(\xi \trgl \pP)
&=&\f^\be (\xi \trgl \pP) S(\f_\be)\v^{-1}
=\f^\be\xi_1\pP S(\xi_2)S(\f_\be)\v^{-1}
=\f^\be\xi_1\of^{\,\ga}\f^\de \pP S(\f_\be\xi_2\of_\ga\f_\de)\v^{-1}\nn\\
&=&
\xi_{1_\FF}\f^\de \pP S(\f_\de)\v^{-1}\v S(\xi_{2_\FF})\v^{-1}
=\xi_{1_\FF}D_\FF(\pP)
S^\FF(\xi_{2_\FF})\nn\\
&=&
\xi\trgl_{\FF}D_\FF(\pP)~,\label{DxitrglP}
\ena
where in the third equality we inserted $1\otimes 1=\FF^{-1}\FF$.
\end{proof}

\sk
Let $\Hom_\bbK(V,W)$ denote the space of $\bbK$-linear maps (morphisms) form $V$
to $W$. Similarly to $\End_\bbK(V)\in \AAAlg{H\!\!}{}{}$ we have 
$\Hom_\bbK(V,W)\in \MMMod{H\!\!}{}{}$, and we can consider the
quantizations $(\Hom_\bbK(V,W)_\st,\RA)\in \MMMod{H^\FF\!\!\!\!}{}{}$, 
and
$(\Hom_\bbK(V_\st,W_\st), \RA_\FF)\in \MMMod{H^\FF\!\!\!\!}{}{}$; here
we have explicitly written the adjoint $H^\FF$-actions carried by the
two modules. 
Then, as in Theorem \ref{PDP0}, the quantization map is an isomorphisms between these two
$H^\FF$-modules.  Also the composition $P\circ Q$ of two $\bbK$-linear
maps $Q : Z\to V$ and $P: V\to W$ can be deformed in the
$\star$-composition $P\circ_\st
Q=(\of^\al\RA P)\circ (\of_\al\RA Q)$. 
\sk
These data have a categorical description.
We have three categories: $(\MMMod{H\!\!}{}{}, \circ, \RA)$,
and the twisted ones $(\MMMod{H^\FF\!\!\!\!}{}{}, \circ_\st, \RA)$ and $(\MMMod{H^\FF\!\!\!\!}{}{}, \circ,
\RA_\FF)$.
\sk
In $(\MMMod{H\!\!}{}{}, \circ, \RA)$, objects are $H$-modules
and morphisms are $\bbK$-linear maps with their usual composition.
These maps are not $H$-equivariant but carry a specific  $H$-action
$\RA$, the one in (\ref{xiadjactP}), that is canonically lifted from
the $H$-action on the (source
and target) modules. 
This action is compatible with composition of  morphism, for all
$\xi\in H$, $P: V\to W$, $Q: Z\to V$, $\xi\RA(P\circ Q)=(\xi_1\RA
P)\circ (\xi_2\RA Q)$.

\sk
In $(\MMMod{H^\FF\!\!\!\!}{}{}, \circ_\st, \RA)$, objects are $H^\FF$-modules
and morphisms are $\bbK$-linear maps with $\star$-composition.
These maps carry the same $\RA$  action (\ref{xiadjactP}), 
that now is seen as an $H^\FF$-action (this is doable since $H$ and
$H^\FF$ are the same as algebras).
This $H^\FF$-action is  compatible with
$\star$-composition of morphisms, $\xi\RA(P\circ_\st Q)=(\xi_{1_\FF}\RA
P)\circ_\st (\xi_{2_\FF}\RA Q)$.
\sk
In
$(\MMMod{H^\FF\!\!\!\!}{}{}, \circ, \RA_\FF)$, objects are $H^\FF$-modules
and morphisms are $\bbK$-linear maps with their usual composition.
The $H^\FF$-action
$\RA_\FF$ on these maps is  the  (\ref{HFadJact}) one, that is
canonically lifted from the $H^\FF$-action on the (source and target) modules. This action is compatible with composition of morphisms, $\xi\RA_\FF(P\circ Q)=(\xi_{1_\FF}\RA_\FF
P)\circ (\xi_{2_\FF}\RA_\FF Q)$.
\sk
It follows from Theorem \ref{PDP0} that
$(\MMMod{H^\FF\!\!\!}\!{}{}, \circ_\st, \RA)$ and $(\MMMod{H^\FF\!\!\!\!}{}{}, \circ,
\RA_\FF)$ are equivalent categories via the functor that is the
identity on objects and $D_\FF$ on morphisms. Indeed $D_\FF$ satisfies 
(\ref{DPSTQDPDQ}) and (\ref{DHFintertwiner}) for $P,Q$ composable
morphisms.
\sk

\subsection{Twisting $A$-module morphisms}\label{Alinearity}
The findings of the previous section holds also in the subcategory
$\MMMod{H\!\!\!}{A}{A}$ of $H$-modules  $A$-bimodules, where $A$ is
an $H$-module algebra. In this case morphisms are right $A$-linear
maps, we use the notations $\Hom_A(V,W)$ and $\End_A(V)$ for right
$A$-linear morphisms and endomorphisms (where $V,W\in \MMMod{H\!\!\!}{A}{A}$).

The left $A$-module structure of $V$ represents $A$ as right
$A$-module endomorphisms of $V$, for all $a\in A$, $a\mapsto
l_a\in\End_A(V)$, where 
$l_a(v)=a\cdot v$ for all $v\in V$.
Then $End_A(V)$ is an $A$-bimodule by defining, for all
$P\in \End_A(V)$,
\eq
a\cdot P=l_a\circ P~~~,~~~~~~P\cdot a=P\circ l_a~.
\en
This $A$-bimodule structure is compatible with the $H$-module one
given by the adjoint action $\RA$, so that 
$(\End_A(V), \RA)\in \AAAlg{H\!\!\!}{A}{A}$. Similarly 
$(\Hom_A(V,W), \RA)\in \MMMod{H\!\!\!}{A}{A}$.

\begin{Theorem}\label{PDP} 
The map  
\begin{flalign}
\label{DFFEK2}
D_\FF ~:~\End_A(V)_\st~~&\longrightarrow~~~~ \End_{A}(V_\st)\nn\\
 P~~~~&\longmapsto~~~~  D_\FF(P):=(\of^\al\btrgl P)\circ \of_\al\trgl~ 
\end{flalign}
is an isomorphism between the left $H^\FF$-module $A_\star$-bimodule algebras
$(\End_A(V)_\st, \RA)\in \AAAlg{H^\FF\!\!\!\!}{A_\st}{A_\st}$
and $(\End_{A_\st}(V_\st),\RA_\FF)\in
\AAAlg{H^\FF\!\!\!\!}{A_\st}{A_\st}$.
\end{Theorem}

\section{Twisting and tensoring}
\subsection{Quasitriangular Hopf algebras and tensor product of $\bbK$-linear maps}
The tensor product $V\otimes W$
of two left $H$-modules
$V,W\in \MMMod{H\!\!\!}{}{}$ over the Hopf
algebra $H$ is again a left $H$-module,
$V\otimes W\in \MMMod{H\!\!\!}{}{}$. The left $H$-action 
is defined using the coproduct,
for all $\xi\in H$, $v\in V$ and $w\in W$,
\begin{flalign}
\label {eqn:productaction}
\xi\ra (v\otimes w) := (\xi_1\ra v)\otimes (\xi_2\ra w)~,
\end{flalign}
and extended by linearity to all $V\otimes W$.

Given two linear maps $V\stackrel{P}{\longrightarrow}\widetilde V$ and
$W\stackrel{Q}{\longrightarrow}\widetilde W$, the tensor product map
$V\otimes W\stackrel{P\otimes Q}{\longrightarrow}\widetilde V\otimes
\widetilde W$
is defined by, for all $v\in V, w\in W$, 
\eq(P\otimes Q)(v\otimes
w)=P(v)\otimes Q(w)\label{eqn:Ktensorho}
\en and extended by linearity to all $V\otimes W$.

This tensor product is not compatible with the $H$-action, indeed a
short calculation shows that, for all $\xi\in H$,  
$\xi\RA (P\otimes Q)\not=(\xi_1 \RA P)\otimes (\xi_2\RA Q)$; equality
holding in case $Q$ is $H$-equivariant i.e., $\xi\RA Q=\epsi(\xi) Q$.
In order to introduce a tensor product compatible with the $H$-action
we need a {\bf quasitriangular Hopf algebra} $(H,\RR)$. This is a Hopf
algebra $H$
with  an invertible element 
$\RR\in H\otimes H$ (called universal $R$-matrix). 
We use the notation $\RR=R^\al\otimes R_\al$. Let $\RR_{21}=R_\al\otimes R^\al$ then if $\RR_{21}=\RR^{-1}$,  the  quasitriangular Hopf algebra $(H,\RR)$ is
called {\bf triangular}.
Among the
properties of the $\RR$-matrix we recall that it satisfies the Yang-Baxter equation
$
\RR_{12}\RR_{13}\RR_{23}=\RR_{23}\RR_{13}\RR_{12}\,.
$

The $\RR$-matrix induces an isomorphism, called {\bf braiding},
between the tensor product modules  $V\otimes W\in \MMMod{H\!\!\!}{}{}$
and $ W\otimes V\in \MMMod{H\!\!\!}{}{}$,
\eq
\label{eqn:Rflipmap}
\tau_{{\cal R}\;W,V}:  W\otimes V\rightarrow V\otimes W~,~~~
w\otimes v\mapsto \tau_{\RR\; W,V}(w\otimes v)= 
(\oR^\alpha\ra v) \otimes (\oR_\alpha\ra w)~,
\en
where we used the notation
$\RR^{-1}=\oR^\al\otimes\oR_\al$.

The map $\tau_{{\cal R}\;W,V}$ 
is a left $H$-module isomorphisms, i.e., for all $\xi\in H, 
v\in V, w\in W,\,$ 
$\xi \ra(\tau_\RR(w\otimes v))=\tau_\RR(\xi\ra(w\otimes v))$,
or equivalently it is $H$-equivariant 
$
\xi\RA\tau_\RR = \varepsilon(\xi)\,\tau_\RR
$.
Let  $V,W,Z\in \MMMod{H\!\!\!}{}{}$ then it follows from the
$\RR$-matrix properties 
that on the triple tensor product $V\otimes
W\otimes Z$ we have the braid relations. For later purposes we
write them in terms of the inverse braid $\tau^{-1}$,
\begin{subequations}
\begin{flalign}\label{hexagon}
\tau^{-1}_{\RR\,1(23)} =
\tau^{-1}_{\RR\,23}\circ\tau^{-1}_{\RR\,12}\\
\tau^{-1}_{\RR\,(12)3} =
\tau^{-1}_{\RR\,12}\circ\tau^{-1}_{\RR\,23}
\end{flalign}
\end{subequations} 
where $\tau^{-1}_{\RR\,12}$ acts on the first and second  entry of
the triple tensor product, $\tau^{-1}_{\RR\,23}$ on the second and
third, and  $\tau^{-1}_{\RR\,1(23)}$ (or $\tau^{-1}_{\RR\,V, W\otimes Z}$) exchanges the first entry with
the second and third ones.

The first expression for example states that flipping an element $v$ to the right of the element
$w\otimes z$ is the same as first flipping $v$ to the right of $w$ and
then the result to the right of $z$.

\begin{Example}\label{Ex3.1}
The universal enveloping algebra ${U}\gg$ of a Lie algebra $\gg$ is a 
cocommutative Hopf algebra (i.e., for all $\xi\in H$, $\Delta(\xi)=\xi_1\otimes\xi_2=\xi_2\otimes\xi_1$).
Every cocommutative Hopf algebra $H$ has
a triangular structure given by the  $R$-matrix $\RR=1\otimes 1$. 
Let $\FF$ be a twist of this cocommutative Hopf algebra $H$, then the
Hopf algebra $H^\FF$ is triangular with $\RR$-matrix
$\RR^\FF=\FF_{21}\RR\FF^{-1} = \FF_{21}\FF^{-1}$. 
\end{Example}
We can now define the $\RR$-tensor product of $\bbK$-linear maps, see also \cite{Majid:1996kd} Chapter 9.3.
\begin{Definition}\label{defi:Rtensor}
Let $(H,\RR)$ be a quasitriangular Hopf algebra and $V,W,\widetilde{V},\widetilde{W}\in \MMMod{H\!\!\!}{}{}$
be left $H$-modules. The {\bf $\RR$-tensor product} of $\bbK$-linear maps 
is defined by, for
all $P\in\Hom_\bfK(V,\widetilde{V})$ and $Q\in\Hom_\bfK(W,\widetilde{W})$,
\begin{flalign}
\label{eqn:Rtensor}
P\otimes_\RR Q := (P\circ \oR^\alpha\ra\,)\otimes (\oR_\alpha\RA Q)\in \Hom_\bfK(V\otimes W,\widetilde{V}\otimes\widetilde{W})~,
\end{flalign}
where $\otimes$ is defined in (\ref{eqn:Ktensorho}).
\end{Definition}
From the definition it immediately follows that
\eq\label{POQPcircQ}
P\otimes_\RR Q = (P\otimes \id)\circ (
\oR^\alpha\ra\;\otimes \,\oR_\alpha\RA Q)
=(P\otimes_\RR \id)\circ (\id\otimes_\RR Q)~.
\en
The lift of 
$P\in \Hom_\bfK(V,\widetilde{V})$  is simply
$P\otimes_\RR\id=P\otimes \id$, while the lift of $Q$ is 
\eq
\id\otimes_\RR Q=
\oR^\alpha\ra\;\otimes \,\oR_\alpha\RA Q~.\label{idQRRQ}
\en
Use of the braiding map $\tau_{\RR}$ (cf. ({\ref{eqn:Rflipmap}}))
allows us to rewrite the lift $\id\otimes_\RR Q$ acting on $V\otimes W$ in terms of the lift 
$Q\otimes \id$ acting on $W\otimes V$,
$\,\id \otimes_\RR Q = \tau_{\RR}\circ (Q\otimes \id ) \circ
\tau^{-1}_{\RR}~.$

We summarize the properties of the $\RR$-tensor product $\otimes_\RR$ in
the following
\begin{Theorem}\label{RotimesK}
Let $(H,\RR)$ be a quasitriangular Hopf algebra and $V,W,Z,\widetilde{V},\widetilde{W},\widetilde{Z},\widehat{V},\widehat{W}
\in \MMMod{H\!\!\!}{}{}$ be left $H$-modules.
The $\RR$-tensor product is compatible with the left $H$-module structure, i.e.,~for all
$\xi\in H$, $P\in\Hom_\bfK(V,\widetilde{V})$ and $Q\in\Hom_\bfK(W,\widetilde{W})$,
\begin{subequations}
\begin{flalign}
\label{eqn:RtensorHmod}
\xi\RA (P\otimes_\RR Q) = (\xi_1\RA P)\otimes_\RR (\xi_2\RA Q)~.
\end{flalign}
Furthermore, the $\RR$-tensor product is associative, i.e.,~for all
$P\in\Hom_\bfK(V,\widetilde{V})$, $Q\in\Hom_\bfK(W,\widetilde{W})$ and $T\in\Hom_\bfK(Z,\widetilde{Z})$,
\begin{flalign}
\label{eqn:Rtensorass}
\bigl(P\otimes_\RR Q\bigr)\otimes_\RR T = P\otimes_\RR \bigl(Q\otimes_\RR T\bigr)~,
\end{flalign}
and satisfies the composition law, for all $P\in\Hom_\bfK(V,\widetilde{V})$,  $Q\in\Hom_\bfK(W,\widetilde{W})$,
$\widetilde{P}\in\Hom_{\bfK}(\widetilde{V},\widehat{V})$ and $\widetilde{Q}\in\Hom_\bfK(\widetilde{W},\widehat{W})$,
\begin{flalign}
\label{eqn:Rtensorcirc}
\bigl(\widetilde{P}\otimes_\RR\widetilde{Q}\bigr)\circ \bigl(P\otimes_\RR Q\bigr) = \bigl(\widetilde{P}\circ (\oR^\alpha\RA P)\bigr)
\otimes_\RR\bigl((\oR_\alpha\RA \widetilde{Q})\circ Q\bigr)~.
\end{flalign}
\end{subequations}
\end{Theorem}

We have studied the tensor product of modules $\otimes$, and that of
morphisms $\otimes_\RR$ in the category of $H$-modules; that
henceforth we denote by 
$(\MMMod{H\!\!\!}{},\circ,\RA, \otimes,\otimes_\RR)$\footnote{This
  category  is  not quite  a tensor category because of
  (\ref{eqn:Rtensorcirc}) (that shows that $(\otimes, \otimes_\RR)$
  is not a bifunctor).}.

\subsection{Twisting tensor product modules and morphisms}

In the same way we have constructed the category
$(\MMMod{H\!\!\!}{},\circ,\RA, \otimes,\otimes_\RR)$, where  modules are
denoted by $V, W,..$, morphisms by $P, Q, ...$ and their tensor
products by
$V\otimes W$, $P\otimes_\RR Q$, we can construct the category
$(\MMMod{H^\FF\!\!\!\!}{},\circ,\RA_\FF,
\otimes_\st,\otimes_{\RR^\FF})$. We have just to replace the
quasitriangular Hopf algebra $(H,\RR)$ with the quasitriangular one
$(H^\FF,\RR^\FF)$, with universal $\RR$-matrix $\RR^\FF=\FF_{21}\RR\FF^{-1}$.

\sk
In $(\MMMod{H^\FF\!\!\!\!}{},\circ,\RA_\FF,
\otimes_\st,\otimes_{\RR^\FF})$ we denote $H^\FF$-modules by $V_\st,W_\st,...$ (recall however that as
$\bbK$-modules they are the same as $V,W...$), morphisms, that are 
$\bbK$-linear maps, by $P,
Q,...$ (for a $\bbK$-linear map $V_\st\stackrel{P}{\longrightarrow}W_\st$
we have $V=V_\st, W=W_\st$, the $H$- or $H^\FF$-module structure being irrelevant), and their tensor products by $V_\st\otimes_\st W_\st$,
$P\otimes_{\RR^\FF} Q$. 
The notation $V_\st\otimes_\st W_\st$ stresses that  the
$H^\FF$-action on $V_\st\otimes_\st W_\st$ is obtained using the $H^\FF$-coproduct (not the
$H$-one), explicitly $\xi\ra (v\otimes_\st w)=(\xi_{1_\FF}\ra
v)\otimes_\st (\xi_{2_\FF}\ra w)$, for all $\xi\in H^\FF$, $v\in V_\st$, $w\in
W_\st$.
\sk

Given the $H$-modules $V$, $W$ we can then consider the 
$H^\FF$-modules
$V_\st\otimes_\st W_\st$ and $(V\otimes W)_\st$. The $H^\FF$ action on these modules is different  because on $(V\otimes W)_\st$
it is obtained using the $H$-coproduct (not the $H^\FF$-one), recall
(\ref{eqn:productaction}).
It is easy to show that $V_\st\otimes_\st W_\st$ and $(V\otimes
W)_\st$ are isomorphic via the $\bfK$-linear and $H$-equivariant map
\eq\label{phiVWVW}
\varphi_{V,W}:=\FF^{-1}\trgl\,: V_\st\otimes_\star W_\st \to (V\otimes W)_\st~~,~~~
v\otimes_{\st} w\mapsto \varphi_{V,W}(v\otimes_{\star} w)
=(\of^\alpha\ra v)  \otimes (\of_\alpha\ra w)~.
\en
Indeed, $\varphi_{V,W}\bigl(\xi\ra_\FF(v\otimes_\star w)\bigr)=
\xi\ra(\varphi_{V,W}(v\otimes_\star w))\,,$ and
the inverse of $\varphi_{V,W}$ is 
$\varphi^{-1}_{V,W}=\FF\ra$.
\sk
Twist quantization of tensor products of morphisms is then described by the following

\begin{Theorem}\label{theo:promodhomdef}
Let $(H,\RR)$ be a quasitriangular Hopf algebra with twist $\FF\in
H\otimes H$ and  $V,W,\widetilde{V},\widetilde{W}\in \MMMod{H\!\!}{}{}$.
Then for all $P\in \Hom_\bfK(V,\widetilde{V})$
and $Q\in\Hom_\bfK(W,\widetilde{W})$
the following diagram of $\bbK$-linear maps commutes:
\begin{flalign}\label{eqn:promodhomdef}
 \xymatrix{
V_\st\otimes_\star W_\st \ar[d]_-{\varphi^{}_{V,W}}\ar[rrrr]^-{D_\FF(P)\otimes_{\RR^\FF}D_\FF(Q)} & & & &\widetilde{V}\otimes_\star\widetilde{W_\st} \ar[d]^-{^{}\varphi_{\widetilde{V},\widetilde{W}}}\\
(V\otimes W)_\st \ar[rrrr]_-{D_\FF\bigl((\of^\alpha\RA P)\otimes_\RR (\of_\alpha\RA Q)\bigr)}& & & &(\widetilde{V}\otimes\widetilde{W})_\st
}
\end{flalign}
\end{Theorem}

We notice that if $P$, $Q$ are $H$-equivariant maps, then
$\otimes_{\RR^\FF}=\otimes_{\RR}=\otimes$ and the horizontal
maps in (\ref {eqn:promodhomdef}) become simply  $P\otimes Q$. Then (\ref
{eqn:promodhomdef}) with $H$-equivariant maps $P,Q$,
and commutativity of the diagram 
 \begin{flalign}\label{eqn:higheriotast}
 \xymatrix{
V_\star\otimes_{\star} W_\star\otimes_{\star} \Omega_\star
\ar[d]_-{\id_{V_\st}\otimes_\RR
  \varphi_{W,\Omega}}\ar[rrr]^-{\varphi_{V,W}\otimes_\RR\id_{\Omega_\st}}
& & &(V\otimes W)_\star \otimes_{\star} \Omega_\star
\ar[d]^-{\varphi_{(V\otimes W),\Omega}}\\
V_\star\otimes_{\star} (W\otimes \Omega)_\star \ar[rrr]_-{\varphi_{V,(W\otimes\Omega)}} & & &(V\otimes W\otimes \Omega)_\star
}
\end{flalign}
(that easily follows from the twist cocycle condition (\ref{ass}))
show that the categories of representations of the Hopf algebras $H$
and $H^\FF$ are equivalent.

In the more general case of arbitrary $\bfK$-linear maps $P,Q$, commutativity of the diagram
(\ref{eqn:promodhomdef}) relates the categories $(\MMMod{H^\FF\!\!\!\!}{},\circ,\RA_\FF,
\otimes_\st,\otimes_{\RR^\FF})$ and $(\MMMod{H^\FF\!\!\!\!}{},\circ_\st,\RA,
\otimes,\otimes_{\RR_\st})$.
In this latter category
$(\MMMod{H^\FF\!\!\!\!}{},\circ_\st,\RA,
\otimes,\otimes_{\RR_\st})$ we denote modules by $V_\st, W_\st,..$ and morphisms by $P, Q, ...\,$. The
$\otimes$-tensor product of modules in this category 
is by definition the new module $(V\otimes W)_\st$.
The tensor product of $\bbK$-linear maps is given by
$P\otimes_{\RR_\st} Q=(\of^\al\RA P)\otimes_\RR(\of_\al\RA Q)$;
it can be shown to be  just (\ref{POQPcircQ}) with the composition of morphisms $\circ$
replaced by the $\star$-composition $\circ_\st$. [Hint: use (\ref{eqn:RtensorHmod})].

\subsection{Twisting tensor products of $A$-module morphisms}

As in Subsection \ref{Alinearity}, we now consider the subcategory 
$\MMMod{H\!\!\!}{A}{A}$ of $H$-modules  $A$-bimodules, where 
morphisms are right $A$-linear maps. In this subcategory we have the
tensor product  $\otimes_A$. We recall that the
tensor product over $A$ of  the modules $V$ and $W$, denoted $V\otimes_A W$, 
is the quotient  of the $\bbK$-module $V\otimes W$ via the
$\bbK$-submodule 
generated by the elements 
$v\cdot a\otimes w-v\otimes a\cdot w$, for all 
$a\in A,v\in V, w\in W$.
The image of $v\otimes w$ under the canonical projection $\pi :
V\otimes W\to V\otimes_A W$ is denoted by $v\otimes_A w$. 

We further restrict to the case where
$A$ is a  {\bf quasi-commutative algebra}, i.e., by definition
for all $a,\tilde a\in A$,
\eq
a\,\tilde a=(\oR^\al\trgl \tilde a) (\oR_\al\trgl  a)~;\label{QCalg}
\en
and where the $A$-bimodules are {\bf quasi-commutative}, i.e., by definition, for all $a\in A, v\in V$,
\eq\label{QCrightact}
v\cdot a=(\oR^\al\trgl a)\cdot(\oR_\al\trgl v)~.
\en
Notice that in $A$ we have 
$a\,\tilde a=(\oR^\al\trgl \tilde a) (\oR_\al\trgl  a)=
(\oR^\be\oR_\al \trgl  a) (\oR_\be\oR^\al\trgl\tilde  a).$ 
We recall that a  {triangular } Hopf algebra $(H,\RR)$, is a {quasitriangular} one with
$\RR^{-1}=\RR_{21}$ i.e., $\oR^\be\oR_\al \otimes
\oR_\be\oR^\al=1\otimes 1\in H\otimes H$. 
We see that quasi-commutative modules and algebras are very natural in the context
of triangular Hopf algebras.  They naturally arise via twist
quantization of commutative algebras that carry a representation of a
cocommutative Hopf algebra with trivial $\RR$-matrix $\RR=1\otimes 1$,
see Example \ref{Ex3.1} (see also \cite{Aschieri:2005zs}).

It can be shown that 
in quasi-commutative bimodules over a quasi-commutative algebra $A$,
right $A$-linear maps are also {\bf quasi-left $A$-linear},
for all $a\in A, w\in W, Q\in \Hom_A(W,\widetilde W)$,
\eq\label{twlefAlin}
\,Q(a\cdot w)=(\oR^\al\trgl a)\cdot(\oR_\al\RA Q)(w)~.
\en
Use of this property leads to prove that 
the $\otimes_\RR$-tensor product of $\bfK$-linear maps over
$V\otimes W$ of Definition
\ref{defi:Rtensor}  induces an $\otimes_\RR$-tensor product of
$A$-linear maps on the $A$-module $V\otimes_A W$, i.e.,
$V\otimes_A W \stackrel{P\otimes_\RR Q}{\longrightarrow}\,\widetilde V\otimes_A
\widetilde W$ is a well defined right $A$-linear map if $V\stackrel{P}{\longrightarrow} \tilde
  V$,$W\stackrel{Q}{\longrightarrow} \widetilde W$ are right
  $A$-linear and $V,\widetilde V,W,\widetilde W \in 
\MMMod{H\!\!\!}{A}{A}$ are quasi-commutative. 

All the properties of the
$\otimes_\RR$ tensor product of $\bfK$-linear maps of Theorem \ref{RotimesK}  hold for this
induced tensor product of $A$-linear maps.

Also the isomorphism $\varphi_{V,W}: V\otimes_\st W\longrightarrow
(V\otimes W)_\st$ induces the isomorphism on the quotient modules
$\varphi_{V,W}: V\otimes_{A_\st} W\longrightarrow
(V\otimes_A W)_\st$, that satisfies a commutative diagram like
(\ref{eqn:higheriotast}).

\begin{Theorem}\label{theo:prodcondef2} Consider a quasi-commutative algebra $A\in
  \AAAlg{H\!\!}{}{}$. Then Theorem \ref{theo:promodhomdef} holds for right $A$-linear maps on quasi-commutative bimodules in 
$\MMMod{H\!\!\!}{A}{A}$. Just replace in diagram (\ref{eqn:promodhomdef})
the tensor products $\otimes$ and $\otimes_\st$, with  $\otimes_A$ and
$\otimes_{A_\st}$ respectively.
\end{Theorem}

\section{Connections}
\subsection{Twisting connections}

Let $A$ be a unital and associative algebra over $\bfK$. 
A {\bf differential calculus} $\bigl(\Omega^\bullet,\wedge,\dif\bigr)$
over $A$ is a graded algebra 
$\bigl(\Omega^\bullet \!= \!\bigoplus_{n\geq0}
\Omega^n,\wedge\bigr)$ over $\bfK$, where $\Omega^0=A$ has degree
zero, together with 
$\bfK$-linear maps of degree one $\dif:\Omega^n \to \Omega^{n+1}$, satisfying $\dif\circ\dif=0$ and the graded Leibniz rule
\begin{flalign}
 \dif(\theta\wedge\theta^\prime) = (\dif\theta)\wedge\theta^\prime + (-1)^{\deg(\theta)}\,\theta\wedge(\dif\theta^\prime)~,
\end{flalign}
for all $\theta,\theta^\prime\in\Omega^\bullet$, with $\theta$ of
homogeneous degree. 
Because of the Leibniz rule, the $\bfK$-modules $\Omega^n$, $n>0$, are 
$A$-bimodules, i.e.~$\Omega^n\in \MMMod{}{A}{A}$.
As in commutative differential geometry  we call $\Omega^n$ the module of
$n$-forms; we also assume that any $1$-form $\theta\in \Omega :=\Omega^1$ can be written as
$\theta=\sum_i a_i\dif b_i$, with $a_i,b_i\in A$, i.e.~that
exact $1$-forms generate $\Omega$ as a left $A$-module.

Let $\bigl(\Omega^\bullet,\wedge,\dif\bigr)$ be a differential calculus
over $A=\Omega^0$; a {\bf right connection} on a bimodule 
 $V\in \MMMod{}{A}{A}$ is a $\bfK$-linear map $\dd:V\to
V\otimes_A\Omega$, satisfying the right Leibniz rule, for all $v\in V$ and $a\in A$,
\begin{flalign}
\label{eqn:rightcon}
 \dd(v\cdot a) = (\dd v)\cdot a + v\otimes_A \dif a~.
\end{flalign}
We denote by $\Con_A(V)$ the set of all connections on the bimodule
$V\in \MMMod{}{A}{A}$.
Notice that this definition holds also if we just have a right $A$-module $V\in
\MMMod{}{}{A}$; actually all the statements in this subsection hold
in the category of right $A$-modules. 
 $\Con_A(V)$ is an affine space over
$\Hom_A(V,V\otimes_A \Omega)$ because for $\dd\in \Con_A(V)$ and $P\in
Hom_A(V,V\otimes_A \Omega)$ then $\dd+P\in\Con_A(V)$, and any
connection differs from a given one $\dd$ by a morphisms $P$.
\sk
Let $H$ be a Hopf algebra and let $A=\Omega^0$ and 
$\Omega^\bullet$ be  left $H$-module algebras.
The differential calculus $\bigl(\Omega^\bullet,\wedge,\dif\bigr)$ is a  {\bf left $H$-covariant differential calculus} over $A$,
if the $H$-action $\ra$ is degree preserving and
the  differential is equivariant, for all $\xi\in H$ and $\theta\in \Omega^\bullet$,
\begin{flalign}
\label{eqn:equivar}
\xi\ra (\dif\theta) = \dif(\xi\ra \theta)~.
\end{flalign}
Since the $H$-action is degree preserving  we have for all $n\geq 0$,
$\Omega^n\in \MMMod{H\!\!\!}{A}{A}$.
\sk
Given a twist $\FF\in H\otimes H$, a left $H$-covariant
differential calculus $\bigl(\Omega^\bullet,\wedge,\dif\bigr)$ over
$A$ can be quantized (see Theorem \ref{Theorem1}) to yield a left $H^\FF$-covariant differential calculus
$\bigl(\Omega^\bullet,\wedge_\star,\dif\bigr)$
over $A_\star$ (equivariance of the differential $\dif$ implies that $\dif$ is also a differential on $\bigl(\Omega^\bullet,\wedge_\star\bigr)$).

\sk
Let $V\in \MMMod{H\!\!\!}{A}{A}$;
we now briefly outline how given a twist $\FF\in H\otimes H$ of the Hopf algebra
$H$, the quantization map $D_\FF$ leads to an isomorphism
$\Con_A(V)\cong \Con_{A_\star}(V_\star)$ between connections on the
undeformed module $V\in \MMMod{H\!\!\!}{A}{A}$ and on the deformed module
$V_\star\in \MMMod{H^\FF\!\!\!\!}{A_\st}{A_\st}$.

We first observe that the left $H^\FF$-module $A_\st$-bimodule isomorphism 
$
(V\otimes_A\Omega)_\star \stackrel{\varphi^{-1}}{\longrightarrow} V_\star\otimes_{A_\star}\Omega_\star
$
(where for ease of notation we dropped the module indices on $\varphi$)
canonically leads to the isomorphism 
$\Hom_\bbK(V_\star,\, (V\otimes_A\Omega)_\star)
\stackrel{\varphi^{-1}}{\longrightarrow}\Hom_{\bbK}(V_\star,\,V_\star\otimes_{A_\star}\Omega_\star)
\,.
$
Composition of the quantization map 
\eq D_\FF ~:~\Hom_\bbK(V, \,V\otimes_{A}\Omega)_\star\longrightarrow
\Hom_{\bbK}(V_\star,\, (V\otimes_A\Omega)_\star)
\en
with this isomorphism gives the left $H^\FF$-module $A_\st$-bimodule  isomorphism
\begin{flalign}\label{DtildeFHOM}
\widetilde{D}_\FF :=\varphi^{-1}\!\circ\! D_\FF~:~
\Hom_\bbK(V, V\otimes_A \Omega)_\star\longrightarrow
\Hom_{\bbK}(V_\star, V_\star\otimes_{A_\star} \Omega_\star)~.
\end{flalign}
\begin{Theorem}\label{theo:condef}
The isomorphism (\ref{DtildeFHOM})
restricts to the left $H^\FF$-module $A_\st$-bimodule  isomorphism
\eq
\widetilde{D}_\FF~:~
\Hom_A(V, V\otimes_A \Omega)_\star\longrightarrow
\Hom_{A_\star}(V_\star, V_\star\otimes_{A_\star} \Omega_\star)~\label{tildeDFHOMA}~~,~~~~
P~\longmapsto\varphi^{-1} \circ (\of^\al\btrgl P)\circ \of_\al\ra\,~
\en
and to the affine space  isomorphism 
\eq
\widetilde{D}_\FF ~:~
\Con_A(V)\longrightarrow
\Con_{A_\star}(V_\star)~~~,~~~~
\label{tildeDFHOMAproof}\dd~\longmapsto\varphi^{-1} \circ
(\of^\al\btrgl \dd)\circ \of_\al\ra\,~,
\en
where
$\Con_A(V)$ and $\Con_{A_\star}(V_\star)$ are
respectively affine spaces over the isomorphic 
modules
$\Hom_A(V, V\otimes_A \Omega)_\st$ and 
$\Hom_{A_\star}(V_\star, V_\star\otimes_{A_\star} \Omega_\star)$
of right $A$-linear, respectively $A_\st$-linear morphisms.
\end{Theorem}

\subsection{Connections on tensor product modules (sum of connections)
}
Connections on quasi-commutative bimodules can be summed to give
connections on tensor product modules.

Let $(H,\RR)$ be a quasitriangular Hopf algebra.
A  left $H$-covariant differential calculus 
$\bigl(\Omega^\bullet,\wedge,\dif\bigr)$  over $A\in \AAAlg{H\!\!}{}{}$ is called 
{\bf graded quasi-commutative} if the algebra $\Omega^\bullet$ is
graded quasi-commutative, i.e., for all $\theta,\theta^\prime\in \Omega^\bullet$ of
homogeneous degree,
\eq\label{gradedqc}
\theta\wedge\theta^\prime=(-1)^{\deg(\theta)\deg(\theta^\prime)} (\oR^\al\ra\theta^\prime)\wedge
(\oR_\al\ra \theta)~.
\en
It can be shown that any left $H$-covariant differential calculus 
$\bigl(\Omega^\bullet,\wedge,\dif\bigr)$ over a quasi-commutative
algebra $A$ is graded quasi-commutative if the bimodule of one-forms $\Omega$ is 
quasi-commutative (and generates $\Omega^n$ for all $n> 1$). 

\begin{Proposition} 
Let $(H,\RR)$ be a quasi-triangular Hopf algebra and 
$\bigl(\Omega^\bullet,\wedge,\dif\bigr)$ be graded quasi-commutative.
A right
connection $\dd$ on a quasi-commutative $A$-bimodule $W\in\MMMod{H\!\!\!}{A}{A}$
is also a {\bf quasi-left connection}, in the sense that we have the
braided Leibniz rule,
for all $a\in A$ and $w\in W$, 
\eq\label{ltc}
\nabla(a\cdot w)=(\oR^\al\ra a)\cdot (\oR_\al\RA
\nabla)(w)+(\R_\al\ra w)\otimes_A(\R^\al\ra \dif a)~.
\en
\end{Proposition}
\noi This property parallels the quasi-left $A$-linearity property
(\ref{twlefAlin}) of right $A$-linear morphisms.
Notice that if $\dd$ is $H$-equivariant we recover the notion of $A$-bimodule
connection \cite{DuboisViolette:1995hh},
\cite{Madore:2000aq} Section 3.6.
\sk

\begin{Lemma}\label{hexagonA1}
Let $W\in \MMMod{H\!\!\!}{A}{A}$ be quasi-commutative  and 
$\Omega\in  \MMMod{H\!\!\!}{A}{A}$ be the bimodule of one-forms of a 
left $H$-covariant graded quasi-commutative differential calculus  $\bigl(\Omega^\bullet,\wedge,\dif\bigr)$. 
Then the inverse braiding map  $\tau^{-1}_\RR: \Omega\otimes W\rightarrow W\otimes \Omega$ 
(recall definition (\ref{eqn:Rflipmap})) canonically induces a left $H$-module $A$-bimodule
isomorphism on the quotient
\begin{flalign}
\tau^{-1}_\RR: \Omega\otimes_A W\longrightarrow W\otimes_A \Omega ~~,~~~
\theta\otimes_A w\mapsto \tau_\RR^{-1}(\theta\otimes_A w)= (\R_\alpha\ra w) \otimes_A (\R^\alpha\ra \theta)~.
\end{flalign}
This satisfies the braid relation
$\tau^{-1}_{\RR\,1(23)}=\tau^{-1}_{\RR\,23}\circ\tau^{-1}_{\RR\,12}$ that is induced from
the braid relation (\ref{hexagon}).
\end{Lemma}

Given two connections $\dd_V: V\rightarrow V\otimes_A \Omega$ and 
$\dd_W: W\rightarrow W\otimes_A \Omega$ we now construct the connection
$\dd_V\oplus_\RR\dd_W: V\otimes_A W\rightarrow V\otimes_A W
\otimes_A\Omega$ on the tensor product module $V\otimes_A W$.
Since connections are $\bbK$-linear maps and not $A$-linear we
actually have to consider their sum on $V\otimes W$; only later we can
then consider the tensor product $V\otimes_A W$.

\begin{Theorem}{$\!$\bf{(Sum of connections).}}
\label{theo:conplus}
 Let $(H,\RR)$ be a {quasitriangular} Hopf algebra and $A\in \AAAlg{H\!\!}{}{}$, $V,W\in \MMMod{H\!\!\!}{A}{A}$ be quasi-commutative. Let also
 $\bigl(\Omega^\bullet,\wedge,\dif\bigr)$ be a graded
 quasi-commutative left $H$-covariant differential calculus over $A$,
 $\dd_V\in \Con_A(V)$ and
 $\dd_W\in\Con_A(W)$. Consider the $\bfK$-linear map
$\dd_V {\;\widehat{\oplus}_\RR\,} \dd_W \,:\, V\otimes W\to V\otimes_A
W\otimes_A\Omega$ defined by
\begin{flalign}\label{SIGMA}
 \dd_V{\;\widehat{\oplus}_\RR\,} \dd_W:=\tau^{-1}_{\RR\,23}\circ\pi\circ(\dd_V\otimes_\RR\id) 
+\pi\circ( \id\otimes_\RR\dd_W) ~,
\end{flalign}
 where $\pi$ denotes  the projections 
$V\otimes_A\Omega\otimes W\to V\otimes_A\Omega\otimes_A W$
and $V\otimes\Omega\otimes_A W\to V\otimes_A\Omega\otimes_A W$,
and
$\tau^{-1}_{\RR\,23}$
is the inverse
braiding map acting on the second and third entry of the tensor product
$V\otimes_A\Omega\otimes_A W$.

The map $ \dd_V{\;\widehat{\oplus}_\RR\,} \dd_W$ induces a 
connection on the quotient module $ V\otimes_A W$,
\eq \label{eqn:tensorcon}
\dd_V \oplus_\RR \dd_W : V\otimes_A W\to V\otimes_A W\otimes_A\Omega~,
\en
defined by, for all $v\in V, w\in W$, $\,(\dd_V \oplus_\RR \dd_W)
(v\otimes_A w):=(\dd_V{\;\widehat{\oplus}_\RR\,} \dd_W)(v\otimes w)$,
and extended by linearity to all $V\otimes_AW$.
\end{Theorem}
\noi The proof of this theorem relies on   the
braided Leibniz rule (\ref{ltc}).

The properties of the $\otimes_\RR$-tensor product imply that the sum
of connections is compatible with the Hopf algebra action, for all $\xi\in H$,
$\xi\RA (\dd_V\oplus_\RR\dd_W) = (\xi\RA\dd_V) \oplus_\RR (\xi\RA\dd_W)~,
$
and that it is associative, 
\begin{flalign}\label{eqn:tensorconass}
\bigl(\dd_V\oplus_\RR \dd_W\bigr)\oplus_\RR \dd_Z = \dd_V\oplus_\RR\bigl(\dd_W\oplus_\RR\dd_Z\bigr)~.
\end{flalign}
This latter property uses also the braid relation $\tau^{-1}_{\RR\,2(34)}=\tau^{-1}_{\RR\,34}\circ\tau^{-1}_{\RR\,23}$ of Lemma \ref{hexagonA1}.

\sk
We end this section by mentioning connections on dual modules
$V'=\Hom_A(V,A)$ (with $V$  finitely
generated and projective). If
$V\in \MMMod{H\!\!}{}{A}$ then $V^\prime \in \MMMod{H\!\!\!}{A}{}$, and a
right connection $\dd_V\in \Con_A(V)$  induces a canonical {\sl left} connection on $V'$, 
$\dd':V'\longrightarrow \Omega\otimes_A V'$. We denote by $\le
v',v\re$ the evaluation of $v'$ on $v$, then (with obvious abuse of
notation)  $\dd'$ is defined by, for all $v'\in V, v\in
V$,
\eq\le\dd'v',v\re= \dif\le v',v\re-\le v'\otimes_A \id,\dd_Vv\re ~,
\en
and satisfies the left
Leibniz rule $\dd'(av')=\dif a\otimes_A v'+a\dd'(v')$.  If we now
consider a quasi-commutative $A$-bimodule $V\in
\MMMod{H\!\!\!}{A}{A}$ and the  Hopf algebra $H$ is triangular we also have
a canonical right connection $\dd_{V'}\in\Con_{A}(V^\prime)$, defined by
\eq\le\dd_{\!V'}v',v\re=\dif\le v',v\re-\le(\oR^\al\ra v')\otimes_A \id,(\oR_\al\RA \dd_V)v\re
~.\en 
{}For example if  $V$ is the module of vector fields on a  noncommutative manifold
then $V'$ is that of one-forms, and given $\dd\in\Con_A(V)$, we can
then consider $\dd_V\oplus_\RR \dd_{V'}$, the
connection on  covariant and contravariant tensor fields $V\otimes_A V'$.

\sk
\subsection{Twisting sums of connections}
The twist quantized  sum of connections, 
$D_\FF(\dd_V\oplus_\RR\dd_W):(V\otimes_A W)_\star \to
(V\otimes_A W\otimes_A\Omega)_\star$ is by construction a $\bfK$-linear
map. As in the case of Theorems \ref{theo:promodhomdef}  and \ref{theo:prodcondef2}, up to $\varphi$-isomorphisms it is a connection.

\begin{Theorem}\label{theo:prodcondef}
Let $(H,\RR)$ be a quasitriangular Hopf algebra with twist $\FF\in H\otimes H$ and 
$A\in \AAAlg{H\!\!}{}{}$, $V, W\in \MMMod{H\!\!\!}{A}{A}$ be quasi-commutative. Let further 
$\bigl(\Omega^\bullet,\wedge,\dif\bigr)$ be a graded quasi-commutative left $H$-covariant differential calculus 
$\dd_V\in \Con_A(V)$ and $\dd_W\in\Con_A(W)$. Then the following diagram commutes
\begin{flalign}\label{eqn:quantprodcon}
\xymatrix{
V_\star \otimes_{A_\star} W_\star \ar[rrrr]^-{\widetilde{D}_\FF(\dd_V)\oplus_{\RR^\FF} \widetilde{D}_\FF(\dd_W)} \ar[d]_-{\varphi_{V,W}}& & & & V_\star \otimes_{A_\star} W_\star \otimes_{A_\star} \Omega_\star \ar[d]^-{\varphi_{V,W,\Omega}}\\
(V\otimes_A W)_\star \ar[rrrr]_-{D_\FF(\dd_V\oplus_\RR\dd_W)} & & & & (V\otimes_A W\otimes_A\Omega)_\star 
} 
\end{flalign}
where $\varphi_{V,W,\Omega}=\varphi_{V\otimes_A W,\Omega}\circ
(\varphi_{V,W}\otimes_\RR\id_{\Omega_\st})$ is the diagonal in the
commutative square (\ref{eqn:higheriotast}) (with the substitutions
$\otimes_\st\rightarrow \otimes_{A_\st}$).
\end{Theorem}
\noi A special case of the above diagram is when $\RR=1\otimes 1$ and $A$
is commutative. Then the lower horizontal arrow is the quantization of  
a usual sum of connections on $A$-bimodules. The upper 
horizontal arrow is then the sum of noncommutative connections 
on quasi-commutative bimodules in $\MMMod{H^\FF\!\!\!}{A_\st}{A_\st}$,
where the Hopf algebra $H^\FF$ has triangular $\RR$-matrix
$\RR^\FF=\FF_{21}\FF^{-1}$.


\section{Curvature}

\subsection{Curvature of connections and of sum of connections}
A connection $\dd : V\longrightarrow V\otimes_A\Omega$ on a right
$A$-module $V$ can be extended
to a well defined $\bbK$-linear map $\dd : V\otimes_A \Omega^\bullet\longrightarrow
V\otimes_A\Omega^\bullet$ by setting, for all $v\in V, \theta\in
\Omega^\bullet$,
\begin{flalign}\label{eqn:nabliftinduced}
\nabla( v\otimes_A\theta ) = (\nabla v)\wedge \theta + v\otimes_A \dif \theta ~.
\end{flalign}
The curvature $R_\dd$ of the connection $\dd$ is the $\bbK$-linear map defined by
\begin{flalign}
R_\nabla:=\nabla\circ\nabla : V\to V\otimes_A\Omega^2~.
\end{flalign}
It is a standard proof to show that it is a right $A$-linear map, i.e.,~$R_\nabla\in\Hom_A(V,V\otimes_A\Omega^2)$.
\sk
If we have a quasitriangular Hopf algebra $(H,\RR)$, and
$A\in\AAAlg{H\!\!}{}{}$\,,\; $V,\,W\in \MMMod{H\!\!\!}{A}{A}$ are
quasi-commutative, and if
$\bigl(\Omega^\bullet,\wedge,\dif\bigr)$ is a graded quasi-commutative
left $H$-covariant differential calculus, then we can consider the 
sum of two connections $\dd_V\in \Con_A(V)$ and $\dd_W\in\Con_A(W)$.
The corresponding  curvature
$R_{\nabla_V\oplus_\RR\nabla_W}
\in\Hom_A(V\otimes_AW,V\otimes_AW\otimes_A\Omega^2)$ can be shown to satisfy the identity
 \eqa\label{eqn:curvaturesum}
 R_{\nabla_V\oplus_\RR\nabla_W}&\!\!\!=&\!\! \tau^{-1}_{\RR\,23}\circ (R_{\nabla_V}\otimes_\RR \id_W \big)
 + \id_V \otimes_\RR R_{\nabla_W} \\
&& +\,\, (\id_{V\otimes_A W}\otimes_\RR\wedge)\circ \tau^{-1}_{\RR\,23} \circ \Big( \nabla_V \otimes_\RR \nabla_W - (\oR^\al\RA\nabla_V) \otimes_\RR (\oR_\al\RA \nabla_W )\Big)~,\nn
 \ena
 where $R_{\nabla_V}\in\Hom_A(V,V\otimes_A\Omega^2)$ and $R_{\nabla_W}\in\Hom_A(W,W\otimes_A\Omega^2)$  are the curvatures of $\nabla_V$ and $\nabla_W$, respectively.
 The second line in (\ref{eqn:curvaturesum}) is a right $A$-linear map
 even though the single addends are not.  

\sk

We remark that in  case one of the two connections is $H$-equivariant, then
the second line in (\ref{eqn:curvaturesum})  vanishes, and (\ref{eqn:curvaturesum}) shows that
the curvature
of the sum of connections is the sum of the curvatures of the initial connections.


\subsection{Curvature of twisted connections and twisted curvatures}
Let $H$ be a Hopf algebra with twist $\FF\in H\otimes H$ and let 
$A\in\AAAlg{H\!\!}{}{}$, $V\in \MMMod{H\!\!\!}{}{A}$, $\dd\in
\Con_A(V)$. We twist these structures to
$A_\st\in\AAAlg{H^\FF\!\!\!\!}{}{}$, $V_\st\in \MMMod{H^\FF\!\!\!\!}{}{A_\st}$, $\dd_\st=\widetilde{D}_\FF(\dd)\in
\Con_{A_\st}(V_\st)$. The isomorphism $\widetilde D_\FF$ between $\Con_A(V)$ and
$\Con_{A_\st}(V_\st)$ can be shown to lift to an isomorphism between extended connections 
$\dd : V\otimes_A \Omega^\bullet\longrightarrow V\otimes_A\Omega^\bullet$ and
$\dd_\st : V_\st\otimes_{A_\st} \Omega_\st^\bullet\longrightarrow
V_\st\otimes_{A_\st}\Omega_\st^\bullet$, we have
\eq
\dd_\st=\varphi_{V,\Omega^\bullet}^{-1}\circ D_\FF(\dd)\circ\varphi_{V,\Omega^\bullet}~.
\en
We can then express the curvature $R_{\dd_\st}$ of the quantized
connection $\dd_\st={\widetilde{D}_\FF(\dd)}$ in terms of the original
connection $\dd$.
We have (use (\ref{DPSTQDPDQ}))
\begin{flalign}
R_{\nabla_\st} := \nabla_\st\circ \nabla_\st 
=\varphi_{V,\Omega^2}^{-1}\circ
D_\FF(\dd)\circ\varphi_{V,\Omega}\circ \widetilde{D}_\FF(\dd)=
\varphi^{-1}_{V_\st,\Omega^2}\circ D_\FF\big(\nabla\circ_\st
\nabla\big) 
= \widetilde{D}_\FF\big(\nabla\circ_\st \nabla\big)~.
\end{flalign}
Notice that the quantized curvature $\widetilde
D_\FF(R_\dd)=\widetilde D_\FF(\dd\circ\dd)$ 
differs from the curvature of the quantized connection
$R_{{\widetilde{D}_\FF(\dd)}}= \widetilde{D}_\FF\big(\nabla\circ_\st
\nabla\big)$, hence flat connections are in general not mapped
into flat connections. The study of the cohomology of
twisted connections that are flat could lead to new cohomology
invariants or interesting combinations of undeformed ones.

\sk


\section*{Acknowledgements}
It is a pleasure to acknowledge the warm hospitality and the fruitful
discussions experienced at the JW2011 workshop of the Balkan Summer
Institute 2011. Fruitful conversations are acknowledged with
Ugo Bruzzo and Walter van Suijlekom.
The hospitality of CERN Theory Unit where the
present work has been completed is also gratefully acknowledged.

This work is in part supported by the exchange grant 2646
of the  ESF Activity Quantum Geometry and Quantum Gravity, by the
Deutsche Forschungsgmeinschaft through the Research Training Group GRK
1147 Theoretical Astrophysics and Particle Physics 
and by the ERC Advanced Grant no. 226455, Supersymmetry, Quantum
Gravity and Gauge Fields (SUPERFIELDS). 

\bibliographystyle{utphys}
\bibliography{Aschieribibl}


\end{document}